\documentclass[a4paper]{article}%
\usepackage{amsmath}
\usepackage{amsfonts}
\usepackage{amssymb}
\usepackage{graphicx}%
\providecommand{\U}[1]{\protect\rule{.1in}{.1in}}
\newtheorem{theorem}{Theorem}

\newtheorem{corollary}[theorem]{Corollary}

\newtheorem{definition}[theorem]{Definition}

\newtheorem{lemma}[theorem]{Lemma}

\newtheorem{proposition}[theorem]{Proposition}

\newenvironment{proof}[1][Proof]{\noindent\textbf{#1.} }{\ \rule{0.5em}{0.5em}}
\allowdisplaybreaks
\begin{document}

\title{A characterization theorem for the $L^{2}$-discrepancy of integer points \\in dilated polygons }
\author{G. Travaglini
\and M. R. Tupputi}
\date{}
\maketitle

\begin{abstract}
Let $C$ be a convex $d$-dimensional body. If $\rho$ is a large positive
number, then the dilated body $\rho C$ contains $\rho^{d}\left\vert
C\right\vert +\mathcal{O}\left(  \rho^{d-1}\right)  $ integer points, where
$\left\vert C\right\vert $ denotes the volume of $C$. The above error estimate
$\mathcal{O}\left(  \rho^{d-1}\right)  $ can be improved in several cases. We
are interested in the $L^{2}$-discrepancy $D_{C}(\rho)$ of a copy of $\rho C$
thrown at random in $\mathbb{R}^{d}$. More precisely, we consider
\[
D_{C}(\rho):=\left\{  \int_{\mathbb{T}^{d}}\int_{SO(d)}\left\vert
\operatorname*{card}\left(  \left(  \rho\sigma(C)+t\right)  \cap\mathbb{Z}%
^{d}\right)  -\rho^{d}\left\vert C\right\vert \right\vert ^{2}d\sigma
dt\right\}  ^{1/2}\ ,
\]
where $\mathbb{T}^{d}=$ $\mathbb{R}^{d}/\mathbb{Z}^{d}$ is the $d$-dimensional
flat torus and $SO\left(  d\right)  $ is the special orthogonal group of real
orthogonal matrices of determinant $1$.

An argument of D. Kendall shows that $D_{C}(\rho)\leq c\ \rho^{(d-1)/2}$. If
$C$ also satisfies the reverse inequality $\ D_{C}(\rho)\geq c_{1}%
\ \rho^{(d-1)/2}$, we say that $C$ is $L^{2}$\emph{-regular}. L. Parnovski and
A. Sobolev proved that, if $d>1$, a $d$-dimensional unit ball is $L^{2}%
$-regular if and only if $d\not \equiv 1\ (\operatorname{mod}4)$.

In this paper we characterize the $L^{2}$-regular convex polygons. More
precisely we prove that a convex polygon is not $L^{2}$-regular if and only if
it can be inscribed in a circle and it is symmetric about the centre.

\medskip\bigskip



\medskip

\end{abstract}

\section{Introduction}

We identify the $d$-dimensional flat torus $\mathbb{T}^{d}=\mathbb{R}%
^{d}/\mathbb{Z}^{d}$ with the unit cube $\left[  -\frac{1}{2},\frac{1}%
{2}\right)  ^{d}$ and we recall that a sequence $\{t_{j}\}_{j=1}^{+\infty
}\subset\mathbb{T}^{d}$ is \textit{uniformly distributed} if one of the
following three equivalent conditions is satisfied: (i) for every
$d$-dimensional box $I\subset\left[  -\frac{1}{2},\frac{1}{2}\right)  ^{d}$
with volume $|I|$,%
\[
\lim_{N\rightarrow+\infty}\frac{1}{N}\,\operatorname*{card}\{t_{j}\in I:1\leq
j\leq N\}=\left\vert I\right\vert \ ;
\]
(ii) for every continuous function $f$ on $\mathbb{T}^{d}$,%
\[
\lim_{N\rightarrow+\infty}\frac{1}{N}\sum_{j=1}^{N}f(t_{j})=\int
_{\mathbb{T}^{d}}f\left(  t\right)  \ dt\ ;
\]
and (iii) for every $0\neq k\in\mathbb{Z}^{d}$,%
\[
\lim_{N\rightarrow+\infty}\frac{1}{N}\sum_{j=1}^{N}e^{2\pi ik\cdot t_{j}%
}=0\ ,
\]
where \textquotedblleft$\cdot$\textquotedblright\ denotes the $d$-dimensional
inner product.

The concept of uniform distribution and the defining properties given above go
back to a fundamental paper written one hundred years ago by H. Weyl \cite{W};
see \cite{KN} for the basic reference on uniformly distributed sequences.
Observe that the above definition does not show the quality of a uniformly
distributed sequence. In the late thirties J. van der Corput coined the term
\textit{discrepancy}: let $\mathfrak{D}_{N}:=\{t_{j}\}_{j=1}^{N}$ be a
sequence of $N$ points in $\mathbb{T}^{d}$, henceforth called a
\textit{distribution} (of $N$ points), and let
\[
D\left(  \mathfrak{D}_{N}\right)  :=\sup_{I\subset\mathbb{T}^{d}}\left\vert
\operatorname*{card}\left(  \{t_{j}\}_{j=1}^{N}\cap I\right)
-N|I|\right\vert
\]
be the (non normalized) discrepancy associated with $\mathfrak{D}_{N}$ with
respect to the $d$-dimensional boxes $I$ in $\mathbb{T}^{d}$. There are
different approaches to define a discrepancy that measures the quality of a
distribution of points; see e.g. \cite{BC88,CST,DT,KN,Mat,DP} for an
introduction of discrepancy theory. See \cite{BCCGST,CF,DG,DHRZ} for the
connections of discrepancy to energy and numerical integration.

Throughout this paper we shall denote by $c,c_{1},\ldots$ positive constants
which may change from step to step.

\bigskip

K. Roth \cite{Roth} proved the following lower estimate: for every
distribution $\mathfrak{D}_{N}$ of $N$ points in $\mathbb{T}^{2}$, we have
\begin{equation}
\int_{\mathbb{T}^{2}}\left\vert \operatorname*{card}(\mathfrak{D}_{N}\cap
I_{x,y})-Nxy\right\vert ^{2}\ dxdy\geq c\ \log N\ , \label{roth}%
\end{equation}
where $I_{x,y}:=[0,x]\times\lbrack0,y]$ and $0\leq x,y<1$. This yields
$D(\mathfrak{D}_{N})\geq c\ \log^{1/2}N$. H. Davenport \cite{Dav} proved that
the estimate (\ref{roth}) is sharp.

W. Schmidt \cite{Scmdt} investigated the discrepancy with respect to discs.
His results were improved and extended, independently, by J. Beck \cite{Beck}
and H. Montgomery \cite{mont}: for every convex body $C\subset\left[
-\frac{1}{2},\frac{1}{2}\right)  ^{d}$ of diameter less than one and for every
distribution $\mathfrak{D}_{N}$ of $N$ points in $\mathbb{T}^{d}$, one has%
\begin{equation}
\int_{0}^{1}\int_{SO(d)}\int_{\mathbb{T}^{d}}\left\vert \operatorname*{card}%
(\mathfrak{D}_{N}\cap(\lambda\sigma(C)+t))-\lambda^{d}N|C|\right\vert
^{2}\ dtd\sigma d\lambda\geq c\ N^{(d-1)/d}\ . \label{bm}%
\end{equation}
This relation implies that for every distribution $\mathfrak{D}_{N}$ there
exists a translated, rotated, and dilated copy $\overline{C}$ of a given
convex body $C\subset\left[  -\frac{1}{2},\frac{1}{2}\right)  ^{d}$ having
diameter less than one, such that
\[
\left\vert \operatorname*{card}(\mathfrak{D}_{N}\cap\overline{C}%
)-N|\overline{C}|\right\vert \geq c\ N^{(d-1)/(2d)}\ .
\]
J. Beck and W. Chen \cite{BC} proved that (\ref{bm}) is sharp. Indeed, they
showed that for every positive integer $N$ there exists a distribution
$\widetilde{\mathfrak{D}}_{N}\subset\mathbb{T}^{d}$ satisfying%
\begin{equation}
\int_{SO(d)}\int_{\mathbb{T}^{d}}\left\vert \operatorname*{card}%
(\widetilde{\mathfrak{D}}_{N}\cap C)-N|C|\right\vert ^{2}\ dtd\sigma\leq
c\ N^{(d-1)/d}\ . \label{pallino discr}%
\end{equation}
This distribution $\widetilde{\mathfrak{D}}_{N}$ can be obtained either by
applying a probabilistic argument or by reduction to a lattice point problem;
see \cite{BGT, CHEN,CT,T} for a comparison of probabilistic and deterministic results.

\bigskip

In the following, we shall consider bounds for the integral in
(\ref{pallino discr}) for distributions of $N$ points that are restrictions of
a shrunk integer lattice to the unit cube $\left[  -\frac{1}{2},\frac{1}%
{2}\right)  ^{d}$. Due to an argument in \cite[p. 3533]{BIT} that also extends
to higher dimensions, we may assume that $N$ is a $d$th power $N=M^{d}$ for a
positive integer $M$. More precisely, we consider distributions%

\[
\mathfrak{D}_{N}:=\left(  \frac{1}{N^{1/d}}\mathbb{Z}^{d}\right)  \cap\left[
-\frac{1}{2},\frac{1}{2}\right)  ^{d}\ .
\]
Given a convex body $C\subset\left[  -\frac{1}{2},\frac{1}{2}\right)  ^{d}$ of
diameter less than one, we then have%
\begin{equation}
\operatorname*{card}\left(  \mathfrak{D}_{N}\cap C\right)  -N\left\vert
C\right\vert =\operatorname*{card}\left(  \mathbb{Z}^{d}\cap N^{1/d}C\right)
-N\left\vert C\right\vert \ . \label{N ro lattice}%
\end{equation}
Estimation of the RHS in (\ref{N ro lattice}) is a classical lattice point
problem. Results concerning lattice points are extensively used in different
areas of pure and applied mathematics; see, for example, \cite{GL,H,K}.

For the definition of a suitable discrepancy function, we change the discrete
dilation $N^{1/d}$ in (\ref{N ro lattice}) to an arbitrary dilation $\rho
\geq1$ and replace the convex body $C$ in (\ref{N ro lattice}) with a
translated, rotated and then dilated copy $\rho\sigma\left(  C\right)  +t$,
where $\sigma\in SO\left(  d\right)  $ and $t\in\mathbb{T}^{d}$. Thus the
discrepancy%
\[
D_{C}^{\rho}(\sigma,t):=\operatorname*{card}(\mathbb{Z}^{d}\cap(\rho
\sigma(C)+t))-\rho^{d}|C|=\sum_{k\in\mathbb{Z}^{d}}\chi_{\rho\sigma
(C)+t}(k)-\rho^{d}|C|
\]
is defined as the difference between the number of integer lattice points in
the set $\rho\sigma\left(  C\right)  +t$ and its volume $\rho^{d}\left\vert
C\right\vert $ (here, $\chi_{A}$ denotes the characteristic function for the
set $A$). It is easy to see (e.g., \cite{BGT}) that the periodic function
$t\mapsto D_{C}^{\rho}(\sigma,t)$ has the Fourier series expansion%
\begin{equation}
\rho^{d}\sum_{0\neq m\in\mathbb{Z}^{d}}\widehat{\chi_{\sigma\left(  C\right)
}}\left(  \rho m\right)  \ e^{2\pi im\cdot t}\ . \label{fourier D}%
\end{equation}
D. Kendall \cite{Kendall} seems to have been the first to realize that
multiple Fourier series expansions can be helpful in certain lattice point
problems. Using our notation, he proved that for every convex body
$C\subset\mathbb{R}^{d}$ and $\rho\geq1$%
\begin{equation}
\Vert D_{C}^{\rho}\Vert_{L^{2}\left(  SO(d)\times\mathbb{T}^{d}\right)  }\leq
c\ \rho^{(d-1)/2}\ . \label{kend}%
\end{equation}
This also follows from more recent results in \cite{Pod91} and \cite{BHI} as
demonstrated next. Given a convex body $C\subset\mathbb{R}^{d}$, we define the
(spherical) average decay of $\widehat{\chi_{C}}$ as%
\[
\left\Vert \widehat{\chi_{C}}\left(  \rho\cdot\right)  \right\Vert
_{L^{2}\left(  \Sigma_{d-1}\right)  }:=\left\{  \int_{\Sigma_{d-1}}\left\vert
\widehat{\chi_{C}}\left(  \rho\tau\right)  \right\vert ^{2}\ d\tau\right\}
^{1/2}\ ,
\]
where $\Sigma_{d-1}:=\left\{  t\in\mathbb{R}^{d}:\left\vert t\right\vert
=1\right\}  $ and $\tau$ is the rotation invariant normalized measure on
$\Sigma_{d-1}$. Extending an earlier result of A. Podkorytov \cite{Pod91}, L.
Brandolini, S. Hofmann, and A. Iosevich \cite{BHI} proved that
\begin{equation}
\left\Vert \widehat{\chi_{C}}\left(  \rho\cdot\right)  \right\Vert
_{L^{2}\left(  \Sigma_{d-1}\right)  }\leq c\ \rho^{-\left(  d+1\right)  /2}\ .
\label{bhi}%
\end{equation}
By applying the Parseval identity to the Fourier series (\ref{fourier D}) of
the discrepancy function, we obtain Kendall's result (\ref{kend}); i.e.,%
\begin{align}
\Vert D_{C}^{\rho}\Vert_{L^{2}\left(  SO(d)\times\mathbb{T}^{d}\right)  }^{2}
&  =\rho^{2d}\sum_{0\neq k\in\mathbb{Z}^{d}}\int_{SO\left(  d\right)
}\left\vert \widehat{\chi_{\sigma\left(  C\right)  }}\left(  \rho k\right)
\right\vert ^{2}\ d\sigma\label{parsev}\\
&  \leq c\ \rho^{2d}\sum_{0\neq k\in\mathbb{Z}^{d}}\left\vert \rho
k\right\vert ^{-\left(  d+1\right)  }\leq c_{1}\ \rho^{d-1}\ .\nonumber
\end{align}
We are interested in the reversed inequality%
\begin{equation}
\Vert D_{C}^{\rho}\Vert_{L^{2}\left(  SO(d)\times\mathbb{T}^{d}\right)  }%
^{2}\geq c_{1}\ \rho^{d-1}\ , \label{reverse}%
\end{equation}
which, as we shall see, may or may not hold. To understand this, let us assume
that (\ref{bhi})\ can be reversed%
\begin{equation}
\left\Vert \widehat{\chi_{C}}\left(  \rho\cdot\right)  \right\Vert
_{L^{2}\left(  \Sigma_{d-1}\right)  }\geq c_{1}\ \rho^{-\left(  d+1\right)
/2}\ . \label{simplex reverse}%
\end{equation}
This relation (\ref{simplex reverse}) is true for a simplex (see \cite[Theorem
2.3]{BCT}) but it is not true for every convex body (see the next section).

The following result was proved in \cite[Proof of Theorem 3.7]{BCT}.

\begin{proposition}
\label{prop}Let $C$ in $\mathbb{R}^{d}$ be a convex body which satisfies
(\ref{simplex reverse}). Then $C$ satisfies (\ref{reverse}).
\end{proposition}

\begin{proof}
Indeed,%
\begin{align}
\Vert D_{C}^{\rho}\Vert_{L^{2}\left(  SO(d)\times\mathbb{T}^{d}\right)  }^{2}
&  =\rho^{2d}\sum_{0\neq k\in\mathbb{Z}^{d}}\int_{SO\left(  d\right)
}\left\vert \widehat{\chi_{\sigma\left(  C\right)  }}\left(  \rho k\right)
\right\vert ^{2}\ d\sigma\label{dalbassocondecad}\\
&  \geq c\ \rho^{2d}\int_{SO\left(  d\right)  }\left\vert \widehat
{\chi_{\sigma\left(  C\right)  }}\left(  \rho k^{\prime}\right)  \right\vert
^{2}\ d\sigma\geq c_{1}\ \rho^{d-1}\ ,\nonumber
\end{align}
where $k^{\prime}$ is any non-zero element in $\mathbb{Z}^{d}$.
\end{proof}

We are going to see that (\ref{reverse}) does not imply (\ref{simplex reverse}).

\section{$L^{2}$-regularity of convex bodies}

We say that a convex body $C\subset\mathbb{R}^{d}$ is $L^{2}$\emph{-regular
}if there exists a positive constant $c_{1}$ such that
\begin{equation}
c_{1}\ \rho^{(d-1)/2}\leq\Vert D_{C}^{\rho}\Vert_{L^{2}\left(  SO(d)\times
\mathbb{T}^{d}\right)  } \label{L2reg}%
\end{equation}
(by (\ref{kend}) we already know that $\Vert D_{C}^{\rho}\Vert_{L^{2}\left(
SO(d)\times\mathbb{T}^{d}\right)  }\leq c_{2}$ $\rho^{(d-1)/2}$ for some
$c_{2}>0$). If (\ref{L2reg}) fails we say that $C$ is $L^{2}$\emph{-irregular}.

Let $d>1$. L. Parnovski and A. Sobolev \cite{PS} proved that the
$d$-dimensional ball $B_{d}:=\left\{  t\in\mathbb{R}^{d}:\left\vert
t\right\vert \leq1\right\}  $ is $L^{2}$-regular if and only if $d\not \equiv
1\ (\operatorname{mod}4)$.

More generally, it was proved \cite{BCGT} that if $C\subset\mathbb{R}^{d}$
($d>1$) is a convex body with smooth boundary, having everywhere positive
Gaussian curvature, then (i) if $C$ is not symmetric about a point, or if
$d\not \equiv 1\left(  \operatorname{mod}4\right)  $, then $C$ is $L^{2}%
$-regular; (ii) if $C$ is symmetric about a point and if $d\equiv1\,\left(
\operatorname{mod}4\right)  $ then $C$ is $L^{2}$-irregular.

L. Parnovski and N. Sidorova \cite{psi} studied the above problem for the
non-convex case of a $d$-dimensional annulus ($d>1$). They provided a complete
answer in terms of the width of the annulus.

In the case of a polyhedron $P$, inequality (\ref{kend}) was extended to
$L^{p}$ norms in \cite{BCT}: for any $p>1$ and $\rho\geq1$ we have
\[
\Vert D_{P}^{\rho}\Vert_{L^{p}\left(  SO(d)\times\mathbb{T}^{d}\right)  }\leq
c_{p}\ \rho^{(d-1)(1-1/p)}%
\]
and, specifically for simplices $S$, one has%
\[
c_{p}^{\prime}\ \rho^{(d-1)(1-1/p)}\leq\Vert D_{S}^{\rho}\Vert_{L^{p}\left(
SO(d)\times\mathbb{T}^{d}\right)  }\leq c_{p}\ \rho^{(d-1)(1-1/p)}\ .
\]
In particular, this implies that the $d$-dimensional simplices are $L^{2}$-regular.

For the planar case it was proved in \cite[Theorem 6.2]{BRT} that every convex
body with piecewise $C^{\infty}$ boundary that is not a polygon is $L^{2}$-regular.

Related results can be found in \cite{BCT,CT,ko}.

Until now no example of a $L^{2}$-irregular polyhedron has been found.

We are interested in identifying the $L^{2}$-regular convex polyhedrons. In
this paper we give a complete answer for the planar case.

\bigskip

Let us first compare the $L^{2}$-regularity for a disc $B\subset\mathbb{R}%
^{2}$ and a square $Q\subset\mathbb{R}^{2}$. Their characteristic functions
$\chi_{B}$ and $\chi_{Q}$ do not satisfy (\ref{simplex reverse}). Indeed,
$\widehat{\chi_{B}}\left(  \xi\right)  =\left\vert \xi\right\vert ^{-1}%
J_{1}\left(  2\pi\left\vert \xi\right\vert \right)  $, where $J_{1}$ is the
Bessel function (see e.g. \cite{T}). Then the zeroes of $J_{1}$ yield an
increasing diverging sequence $\left\{  \rho_{u}\right\}  _{u=1}^{\infty}$
such that
\[
\left\Vert \widehat{\chi_{C}}\left(  \rho_{u}\cdot\right)  \right\Vert
_{L^{2}\left(  \Sigma_{1}\right)  }=0\ .
\]

Less obvious is the fact that the inequality $\left\Vert \widehat{\chi_{C}%
}\left(  \rho\cdot\right)  \right\Vert _{L^{2}\left(  \Sigma_{1}\right)  }\geq
c\ \rho^{-3/2}$ fails for a square $Q$: it was observed in \cite{BCT} the
existence of a positive constant $c$ such that, for every positive integer
$n$, one has%
\begin{equation}
\left\Vert \widehat{\chi_{Q}}\left(  n\cdot\right)  \right\Vert _{L^{2}\left(
\Sigma_{1}\right)  }\leq c\ n^{-7/4}\ . \label{7quarti}%
\end{equation}
For completeness we write the short proof of (\ref{7quarti}). Indeed, let
$\ Q=\left[  -\frac{1}{2},\frac{1}{2}\right]  ^{2}$ and let $n$ be a positive
integer. Let $\Theta:=\left(  \cos\theta,\sin\theta\right)  $. Then an
explicit computation of $\widehat{\chi_{Q}}$ yields%
\begin{align*}
&  \int_{0}^{2\pi}\left\vert \widehat{\chi_{Q}}(n\Theta)\right\vert
^{2}~d\theta=8\int_{0}^{\pi/4}\left\vert \dfrac{\sin(\pi n\cos\theta)}{\pi
n\cos\theta}\frac{\sin(\pi n\sin\theta)}{\pi n\sin\theta}\right\vert
^{2}~d\theta\\
&  \leq c\,\frac{1}{n^{4}}\int_{0}^{\pi/4}\left\vert \dfrac{\sin(\pi
n\cos\theta)}{\sin\theta}\right\vert ^{2}~d\theta=c\,\frac{1}{n^{4}}\int
_{0}^{\pi/4}\left\vert \dfrac{\sin(\pi n\left(  1-2\sin^{2}\left(
\theta/2\right)  \right)  )}{\sin\theta}\right\vert ^{2}~d\theta\\
&  \leq c^{\prime}\,\frac{1}{n^{4}}\int_{0}^{\pi/4}\left\vert \sin(2\pi
n\sin^{2}\left(  \theta/2\right)  )\right\vert ^{2}\theta^{-2}~d\theta\\
&  \leq c^{\prime\prime}\,\frac{1}{n^{4}}\int_{0}^{n^{-1/2}}n^{2}\theta
^{2}~d\theta+c^{\prime\prime}\ \frac{1}{n^{4}}\int_{n^{-1/2}}^{\pi/4}%
\theta^{-2}~d\theta\leq c^{\prime\prime\prime}\ n^{-7/2}\;.
\end{align*}

Then $B$ and $Q$ \textit{may} be $L^{2}$-irregular.

On the one hand it is known that a disc $B$ is $L^{2}$-regular (see \cite{PS}
or \cite[Theorem 6.2]{BRT}), so that (\ref{reverse}) does not imply
(\ref{simplex reverse}). On the other hand we shall prove in this paper that
$Q$ is $L^{2}$-irregular.

The $L^{2}$-irregularity of the square $Q$ is shared by each member of the
family of polygons described in the following definition.

\begin{definition}
\label{def Gotic B}Let $\mathfrak{P}$ be the family of all convex polygons in
$\mathbb{R}^{2}$ which can be inscribed in a circle and are symmetric about
the centre.
\end{definition}

\section{Statements of the results}

We now state our main result.

\begin{theorem}
\label{maintheorem}A convex polygon $P$ is $L^{2}$-regular if and only if
$P\notin\mathfrak{P}$.
\end{theorem}

The \textquotedblleft only if\textquotedblright\ part is a consequence of the
following more precise result.

\begin{proposition}
\label{propeps}If $P\in\mathfrak{P}$, then for every $\varepsilon>0$ there is
an increasing diverging sequence $\left\{  \rho_{u}\right\}  _{u=1}^{\infty}$
such that%
\[
\Vert D_{P}^{\rho_{u}}\Vert_{L^{2}\left(  SO(2)\times\mathbb{T}^{2}\right)
}\leq c_{\varepsilon}\ \rho_{u}^{1/2}\log^{-1/\left(  32+\varepsilon\right)
}(\rho_{u})\ .
\]

\end{proposition}

Theorem \ref{maintheorem} above and \cite[Theorem 6.2]{BRT} yield the
following more general result.

\begin{corollary}
Let $C$ be a convex body in $\mathbb{R}^{2}$ having piecewise smooth boundary.
Then $C$ is not $L^{2}$-regular if and only if it belongs to $\mathfrak{P}$.
\end{corollary}

The following result shows that Theorem \ref{maintheorem} is essentially sharp.

\begin{proposition}
\label{stimaepsilon} For every $P\in\mathfrak{P}$, for $\varepsilon>0$
arbitrary small, and for any $\rho$ large enough,%
\[
\Vert D_{P}^{\rho}\Vert_{L^{2}(SO(2)\times\mathbb{T}^{2})}\geq c_{\varepsilon
}\ \rho^{1/2-\varepsilon}\ ,
\]
where $c_{\varepsilon}$ is independent of $\rho$.
\end{proposition}

The \textquotedblleft if\textquotedblright\ part of Theorem \ref{maintheorem}
is a consequences of the following three lemmas.

\begin{lemma}
\label{lemma1}Let $P$ in $\mathbb{R}^{2}$ be a polygon having a side not
parallel to any other side. Then $P$ is $L^{2}$-regular.
\end{lemma}

\begin{lemma}
\label{lemma2}Let $P$ in $\mathbb{R}^{2}$ be a convex polygon with a pair of
parallel sides having different lengths. Then $P$ is $L^{2}$-regular.
\end{lemma}

\begin{lemma}
\label{lemma3} Let $P$ in $\mathbb{R}^{2}$ be a convex polygon which cannot be
inscribed in a circle. Then $P$ is $L^{2}$-regular.
\end{lemma}

\section{ Notation and preliminary arguments}

In the remainder of the paper, a polygon $P$ is given by its vertex set
$\{P_{h}\}_{h=1}^{s}$, where it is assumed that the numbering indicates
counterclockwise ordering of the vertices;\ we write $P\sim\{P_{h}\}_{h=1}%
^{s}$. For convenience we use periodic labeling; i.e., $P_{h+s},$ $P_{h+2s},$
$\ldots$ refer to the same point $P_{h}$ for $1\leq h\leq s$. For every $h$ let%

\[
\tau_{h}:=\frac{P_{h+1}-P_{h}}{\left\vert P_{h+1}-P_{h}\right\vert }%
\]
be the direction of the oriented side $P_{h}P_{h+1}$ and $\ell_{h}:=\left\vert
P_{h+1}-P_{h}\right\vert $ its length.

For every $h$ let $\nu_{h}$ be the outward unit normal vector corresponding to
the side $P_{h}P_{h+1}$. Let
\[
\mathcal{L}_{h}:=\left\vert P_{h}+P_{h+1}\right\vert
\]
be the length of the vector $P_{h}+P_{h+1}$. Observe that if $\left\vert
P_{h}\right\vert =\left\vert P_{h+1}\right\vert $ (in particular if the
polygon $P$ is inscribed in a circle centred at the origin) then
\[
P_{h}+P_{h+1}=\mathcal{L}_{h}\nu_{h}\ .
\]
We shall always assume $\ell_{h}\geq1$ and $\mathcal{L}_{h}\geq1$.

Let $\nu(s)$ be the outward unit normal vector at a point $s\in\partial P$
which is not a vertex of $P$. By applying Green's formula we see that, for any
$\rho\geq1$, we have%
\begin{align}
&  \widehat{\chi_{P}}(\rho\Theta)\nonumber\\
&  =\int_{P}e^{-2\pi i\rho\Theta\cdot t}\ dt=-\frac{1}{2\pi i\rho}%
\int_{\partial P}e^{-2\pi i\rho\Theta\cdot s}\ \left(  \Theta\cdot
\nu(s)\right)  \ ds\nonumber\\
&  =-\frac{1}{2\pi i\rho}\sum_{h=1}^{s}\ell_{h}\left(  \Theta\cdot\nu
_{h}\right)  \int_{0}^{1}e^{-2\pi i\rho\Theta\cdot\left(  P_{h}+\lambda\left(
P_{h+1}-P_{h}\right)  \right)  }\ d\lambda\nonumber\\
&  =-\frac{1}{4\pi^{2}\rho^{2}}\sum_{h=1}^{s}\frac{\Theta\cdot\nu_{h}}%
{\Theta\cdot\tau_{h}}\ \left[  e^{-2\pi i\rho\Theta\cdot P_{h+1}}-e^{-2\pi
i\rho\Theta\cdot P_{h}}\right] \nonumber\\
&  =-\frac{1}{4\pi^{2}\rho^{2}}\sum_{h=1}^{s}\frac{\Theta\cdot\nu_{h}}%
{\Theta\cdot\tau_{h}}\ e^{-\pi i\rho\Theta\cdot(P_{h+1}+P_{h})}\left[  e^{-\pi
i\rho\Theta\cdot(P_{h+1}-P_{h})}-e^{\pi i\rho\Theta\cdot(P_{h+1}-P_{h}%
)}\right] \nonumber\\
&  =\frac{i}{2\pi^{2}\rho^{2}}\sum_{h=1}^{s}\frac{\Theta\cdot\nu_{h}}%
{\Theta\cdot\tau_{h}}\ e^{-\pi i\rho\mathcal{L}_{h}\Theta\cdot\nu_{h}}\sin
(\pi\rho\ell_{h}\Theta\cdot\tau_{h})\ . \label{divergenza}%
\end{align}

For any $1\leq h\leq s$, let $\theta_{h}\in\lbrack0,2\pi)$ be the angle
defined by
\begin{equation}
\tau_{h}=:(\cos\theta_{h},\sin\theta_{h})\ . \label{thetaj}%
\end{equation}
Hence
\begin{equation}
\nu_{h}=(\sin\theta_{h},-\cos\theta_{h}) \label{nu j}%
\end{equation}
and, if $\Theta:=\left(  \cos\theta,\sin\theta\right)  $,
\[
\Theta\cdot\tau_{h}=\cos(\theta-\theta_{h})\ ,\ \ \ \ \ \Theta\cdot\nu
_{h}=-\sin(\theta-\theta_{h}).
\]
Then (\ref{divergenza}) can be written as%
\[
\widehat{\chi_{P}}(\rho\Theta)=-\frac{i}{2\pi^{2}\rho^{2}}\sum_{h=1}^{s}%
\frac{\sin\left(  \theta-\theta_{h}\right)  }{\cos\left(  \theta-\theta
_{h}\right)  }\ e^{\pi i\rho\mathcal{L}_{h}\sin\left(  \theta-\theta
_{h}\right)  }\sin\left(  \pi\rho\ell_{h}\cos\left(  \theta-\theta_{h}\right)
\right)
\]
and the equality in (\ref{parsev}) yields
\begin{align}
&  \Vert D_{P}^{\rho}\Vert_{L^{2}\left(  SO(2)\times\mathbb{T}^{2}\right)
}^{2}\label{generale}\\
&  =\rho^{4}\sum_{0\neq k\in\mathbb{Z}^{2}}\int_{0}^{2\pi}\left\vert
\widehat{\chi_{P}}(\rho\left\vert k\right\vert \Theta)\right\vert
^{2}\ d\theta\nonumber\\
&  =c\sum_{0\neq k\in\mathbb{Z}^{2}}\frac{1}{|k|^{4}}\nonumber\\
&  \times\int_{0}^{2\pi}\left\vert \sum_{h=1}^{s}\frac{\sin(\theta-\theta
_{h})}{\cos(\theta-\theta_{h})}\ e^{-\pi i\rho\left\vert k\right\vert
\mathcal{L}_{h}\sin\left(  \theta-\theta_{h}\right)  }\sin(\pi\rho\left\vert
k\right\vert \ell_{h}\cos(\theta-\theta_{h}))\right\vert ^{2}d\theta
\ .\nonumber
\end{align}

For $P\in\mathfrak{P}$, relation (\ref{generale}) can be further simplified.
Let $P\in\mathfrak{P}$ have $s=2n$ sides (i.e. $P\sim\{P_{h}\}_{h=1}^{2n}$)
and be inscribed in a circle centered at the origin. Then $P_{h}%
P_{h+1}=-P_{n+h}P_{n+h+1}$ for any $1\leq h\leq n$ and $P_{h+1}+P_{h}%
=\mathcal{L}_{h}\nu_{h}$. Therefore, for every $1\leq h\leq n$,
\[
\tau_{h}=-\tau_{n+h}\ ,\ \ \ \ \ \nu_{h}=-\nu_{n+h}\ ,\ \ \ \ \ \ell_{h}%
=\ell_{n+h}\ ,\ \ \ \ \ \mathcal{L}_{h}=\mathcal{L}_{n+h}\ .
\]
Then the relation (\ref{divergenza}) becomes%
\begin{equation}
\widehat{\chi}_{P}(\rho\Theta)=\frac{1}{\pi^{2}\rho^{2}}\sum_{h=1}^{n}%
\frac{\sin(\theta-\theta_{h})}{\cos(\theta-\theta_{h})}\sin(\pi\rho
\mathcal{L}_{h}\sin(\theta-\theta_{h}))\sin(\pi\rho\ell_{h}\cos(\theta
-\theta_{h}))\ . \label{divergenza2}%
\end{equation}
and the equality in (\ref{parsev}) yields
\begin{align}
&  \Vert D_{P}^{\rho}\Vert_{L^{2}\left(  SO(2)\times\mathbb{T}^{2}\right)
}^{2}\label{final}\\
&  =c\sum_{0\neq k\in\mathbb{Z}^{2}}\frac{1}{|k|^{4}}\nonumber\\
&  \times\int_{0}^{2\pi}\left\vert \sum_{h=1}^{n}\frac{\sin(\theta-\theta
_{h})}{\cos(\theta-\theta_{h})}\sin(\pi\rho|k|\ell_{h}\cos(\theta-\theta
_{h}))\sin(\pi\rho|k|\mathcal{L}_{h}\sin(\theta-\theta_{h}))\right\vert
^{2}d\theta\nonumber\\
&  \leq c\ \sum_{0\neq k\in\mathbb{Z}^{2}}\frac{1}{|k|^{4}}\sum_{h=1}^{n}%
\int_{0}^{\pi/2}\left\vert \frac{\sin(\pi\rho|k|\ell_{h}\sin\theta)}%
{\sin\theta}\sin(\pi\rho|k|\mathcal{L}_{h}\cos\theta)\right\vert ^{2}%
d\theta\ .\nonumber
\end{align}
The last relation holds for every $P\in\mathfrak{P}$ with $2n$ sides.

\section{Proofs}

\begin{proof}
[Proof of Lemma \ref{lemma1}]The proof of Lemma \ref{lemma1} is essentially
the proof of \cite[Theorem 3.7]{BCT}, which is stated for a simplex but the
argument also works for every polyhedron having a face not parallel to any
other face.
\end{proof}

\begin{proof}
[Proof of Lemma \ref{lemma2}]By Lemma \ref{lemma1} we can assume that
$P\sim\{P_{h}\}_{h=1}^{2n}$ is a convex polygon with an even number of sides,
and that for every $h=1,\ldots,n$ the sides $P_{h}P_{h+1}$ and $P_{h+n}%
P_{h+n+1}$ are parallel. Suppose that the length $\ell_{j}$ of the $j$th side
$P_{j}P_{j+1}$ is longer than the length $\ell_{j+n}$ of the opposite side
$P_{j+n}P_{j+n+1}$. Then there exist $0<\varepsilon<1$ and $0<\alpha<1$ such
that%
\begin{equation}
(1+\varepsilon)\frac{\ell_{j+n}}{\ell_{j}}<\alpha\ . \label{alfa}%
\end{equation}
Let $H>1$ be a large constant satisfying
\begin{equation}
\sin\left(  \theta-\theta_{j}\right)  \geq\sqrt{\alpha}\left(  \theta
-\theta_{j}\right)  \ \ \ \ \ \text{if \ }0\leq\theta-\theta_{j}\leq
\frac{1+\varepsilon}{H}\ . \label{alfa A}%
\end{equation}
We further assume (recall $\rho\geq1$)%
\[
\frac{1}{H\pi\rho\ell_{j}}\leq\theta-\theta_{j}\leq\frac{1+\varepsilon}%
{H\pi\rho\ell_{j}}\ .
\]
Observe that (\ref{alfa}) and (\ref{alfa A}) yield%
\begin{align*}
&  |\sin(\pi\rho\ell_{j}\sin\left(  \theta-\theta_{j}\right)  )|-|\sin(\pi
\rho\ell_{j+n}\sin\left(  \theta-\theta_{j+n}\right)  )|\\
&  \geq\sin(\pi\rho\ell_{j}\sqrt{\alpha}\left(  \theta-\theta_{j}\right)
)-\sin(\pi\rho\ell_{j+n}\left(  \theta-\theta_{j+n}\right)  )\geq\frac{\alpha
}{H}-\frac{1+\varepsilon}{H}\frac{\ell_{j+n}}{\ell_{j}}=:a_{j}>0\ .
\end{align*}
Hence%
\begin{align}
&  \left\vert \frac{\sin(\pi\rho\ell_{j}\sin(\theta-\theta_{j}))}{\sin
(\theta-\theta_{j})}\cos(\theta-\theta_{j})e^{-\pi i\rho\Theta\cdot\left(
P_{j+1}+P_{j}\right)  }\right. \label{19}\\
&  \left.  +\frac{\sin(\pi\rho\ell_{j+n}\sin(\theta-\theta_{j}))}{\sin
(\theta-\theta_{j})}\cos(\theta-\theta_{j+n})e^{-\pi i\rho\Theta\cdot\left(
P_{j+n+1}+P_{j+n}\right)  }\right\vert \nonumber\\
&  \geq\frac{|\cos(\theta-\theta_{j})|}{|\sin(\theta-\theta_{j})|}\left(
|\sin(\pi\rho\ell_{j}\sin(\theta-\theta_{j}))|-|\sin(\pi\rho\ell_{j+n}%
\sin(\theta-\theta_{j+n}))|\right) \nonumber\\
&  \geq a_{j}\ \frac{|\cos(\theta-\theta_{j})|}{|\sin(\theta-\theta_{j}%
)|}\ .\nonumber
\end{align}
We use the previous estimates to evaluate the last integral in (\ref{generale}%
) in a neighborhood of $\theta_{j}$ and therefore obtain an estimate from
below of $\Vert D_{P}^{\rho}\Vert_{L^{2}\left(  SO(2)\times\mathbb{T}%
^{2}\right)  }$. By the arguments in \cite[Theorem 2.3]{BCT} or \cite[Lemma
10.6]{T}, the contribution of all the sides $P_{h}P_{h+1}$ (with $h\neq j$ and
$h\neq j+n$) to the term $\Vert D_{P}^{\rho}\Vert_{L^{2}\left(  SO(2)\times
\mathbb{T}^{2}\right)  }$ is $\mathcal{O}\left(  1\right)  $. Then
(\ref{dalbassocondecad}), (\ref{divergenza2}) and (\ref{19}) yield%
\[
\Vert D_{P}^{\rho}\Vert_{L^{2}\left(  SO(2)\times\mathbb{T}^{2}\right)  }%
^{2}\geq c\ \int_{\frac{1}{H\pi\rho\ell_{j}}}^{\frac{1+\varepsilon}{H\pi
\rho\ell_{j}}}\frac{\cos^{2}\theta}{\sin^{2}\theta}\ d\theta+c_{1}\geq
c\int_{\frac{1}{H\pi\rho\ell_{j}}}^{\frac{1+\varepsilon}{H\pi\rho\ell_{j}}%
}\frac{d\theta}{\theta^{2}}+c_{1}\geq c_{2}\ \rho\ .
\]

\end{proof}

\begin{proof}
[Proof of Lemma \ref{lemma3}]We can assume that $P\sim\{P_{h}\}_{h=1}^{2n}$ is
a convex polygon such that for every $h=1,\ldots,n$ the sides $P_{h}P_{h+1}$
and $P_{h+n}P_{h+n+1}$ are parallel and of the same length (that is, $\ell
_{h}=\ell_{h+n}$, $\ \tau_{h}=-\tau_{h+n}$, \ $\nu_{h}=-\nu_{h+n}$). Then we
may assume that $P$ is symmetric about the origin. As $P$ cannot be inscribed
in a circle, there exists an index $1\leq j\leq n$ such that the two opposite
equal and parallel sides $P_{j}P_{j+1}$ and $P_{j+n}P_{j+n+1}$ are not the
sides of a rectangle. Then $P_{j}+P_{j+1}$ is not orthogonal to $P_{j+1}%
-P_{j}$. Let $\phi_{j}\in\lbrack\theta_{j}-\pi,\theta_{j}]$ be defined by
\[
P_{j+1}+P_{j}=\mathcal{L}_{j}(\cos\phi_{j},\sin\phi_{j})\ .
\]
Since $\tau_{j}=(\cos\theta_{j},\sin\theta_{j})$ and $\nu(j)=(\cos(\theta
_{j}-\frac{\pi}{2}),\sin(\theta_{j}-\frac{\pi}{2}))$, see (\ref{thetaj}) and
(\ref{nu j}), we have $\phi_{j}-\theta_{j}\neq-\frac{\pi}{2}$. We put
$\varphi_{j}:=\phi_{j}-\theta_{j}$. Then
\[
\varphi_{j}\in\lbrack-\pi,0]\setminus\{-\frac{\pi}{2}\}\ .
\]
Again we need to find a lower bound for the last integral in (\ref{generale}).
As in the previous proof it is enough to consider%
\begin{align*}
F_{j}(\theta)  &  :=\sum_{h\in\{j,j+n\}}\frac{\sin(\theta-\theta_{h})}%
{\cos(\theta-\theta_{h})}\sin(\pi\rho\ell_{j}\cos(\theta-\theta_{h}))e^{-\pi
i\rho\Theta\cdot(P_{h+1}+P_{h})}\\
&  =\frac{\sin(\theta-\theta_{j})}{\cos(\theta-\theta_{j})}\sin(\pi\rho
\ell_{j}\cos(\theta-\theta_{j}))\left[  e^{-\pi i\rho\mathcal{L}_{j}%
\cos(\theta-\phi_{j})}-e^{\pi i\rho\mathcal{L}_{j}\cos(\theta-\phi_{j}%
)}\right] \\
&  =-2i\ \frac{\sin(\theta-\theta_{j})}{\cos(\theta-\theta_{j})}\sin(\pi
\rho\ell_{j}\cos(\theta-\theta_{j}))\sin(\pi\rho\mathcal{L}_{j}\cos
(\theta-\phi_{j}))\ .
\end{align*}
We write%
\begin{align*}
&  \int_{0}^{2\pi}\left\vert \frac{\sin(\theta-\theta_{j})}{\cos(\theta
-\theta_{j})}\sin(\pi\rho\ell_{j}\cos(\theta-\theta_{j}))\sin(\pi
\rho\mathcal{L}_{j}\cos(\theta-\phi_{j}))\right\vert ^{2}\ d\theta\\
&  =\int_{0}^{2\pi}\left\vert \frac{\sin\left(  \pi\rho\ell_{j}\sin
\theta\right)  }{\sin\theta}\cos\theta\sin(\pi\rho\mathcal{L}_{j}\sin
(\theta-\varphi_{j}))\right\vert ^{2}\ d\theta\ .
\end{align*}
We shall integrate $\theta$ in a neighborhood of $0$ (actually $0\leq
\theta\leq1$ suffices). As for $\varphi_{j}$ we first assume $\varphi_{j}%
\in(-\frac{\pi}{2},0]$. Then $\cos\varphi_{j}>0$ and $\sin\varphi_{j}\leq0$.
Let $0<\gamma<1$ satisfy $\cos\varphi_{j}>\gamma$. In order to prove that
$\left\vert \sin(\pi\rho\mathcal{L}_{j}\sin(\theta-\varphi_{j})\right\vert
\geq c$ we consider two cases.

\noindent\emph{Case 1:} $|\sin(\pi\rho\mathcal{L}_{j}\sin\varphi_{j}%
)|>\gamma/2$.

\noindent We need to bound $\sin(\theta-\varphi_{j})-\left\vert \sin
\varphi_{j}\right\vert $. Since $\sin\varphi_{j}\leq0$ one has%
\[
\frac{\theta}{2}\cos\varphi_{j}+\left[  1-\frac{\theta^{2}}{2}\right]
|\sin\varphi_{j}|\leq\sin\theta\cos\varphi_{j}-\cos\theta\sin\varphi_{j}%
\leq\theta\cos\varphi_{j}+|\sin\varphi_{j}|\ .
\]
Therefore
\begin{equation}
\frac{\theta}{2}\gamma-\frac{\theta^{2}}{2}\leq\sin(\theta-\varphi
_{j})-\left\vert \sin\varphi_{j}\right\vert \leq\theta\ . \label{newast}%
\end{equation}
Let $\rho\geq1$ and assume
\[
\frac{\gamma}{8\pi\rho\mathcal{L}_{j}}\leq\theta\leq\frac{\gamma}{4\pi
\rho\mathcal{L}_{j}}\ .
\]
We recall that $\mathcal{L}_{j}\geq1$. Again we have to estimate $\sin
(\theta-\varphi_{j})-\left\vert \sin\varphi_{j}\right\vert $. By
(\ref{newast}) we have%
\[
0<\frac{\gamma}{16\pi\rho\mathcal{L}_{j}}-\frac{\gamma^{2}}{32(\pi
\rho\mathcal{L}_{j})^{2}}<\sin(\theta-\varphi_{j})-|\sin\varphi_{j}|\leq
\frac{\gamma}{4\pi\rho\mathcal{L}_{j}}\ .
\]
Therefore
\begin{equation}
0<\pi\rho\mathcal{L}_{j}\sin(\theta-\varphi_{j})-\pi\rho\mathcal{L}_{j}%
|\sin\varphi_{j}|\leq\frac{\gamma}{4}\ . \label{astetrig1}%
\end{equation}
Hence the assumption of \emph{Case 1} and (\ref{astetrig1}) yield
\begin{align*}
&  \left\vert \sin(\pi\rho\mathcal{L}_{j}\sin(\theta-\varphi_{j}))\right\vert
\\
&  =\left\vert \sin(\pi\rho\mathcal{L}_{j}\left[  \sin(\theta-\varphi
_{j})+\sin\varphi_{j}\right]  -\pi\rho\mathcal{L}_{j}\sin\varphi
_{j})\right\vert \\
&  =\left\vert \sin\left(  \pi\rho\mathcal{L}_{j}\left[  \sin(\theta
-\varphi_{j})-\left\vert \sin\varphi_{j}\right\vert \right]  \right)
\cos\left(  \pi\rho\mathcal{L}_{j}\sin\varphi_{j}\right)  \right. \\
&  \left.  -\cos\left(  \pi\rho\mathcal{L}_{j}\left[  \sin(\theta-\varphi
_{j})-\left\vert \sin\varphi_{j}\right\vert \right]  \right)  \sin\left(
\pi\rho\mathcal{L}_{j}\sin\varphi_{j}\right)  \right\vert \\
&  \geq\left\vert \sin(\pi\rho\mathcal{L}_{j}\sin\varphi_{j})\right\vert
\left\vert \cos(\pi\rho\mathcal{L}_{j}\sin(\theta-\varphi_{j})-\pi
\rho\mathcal{L}_{j}|\sin\varphi_{j}|)\right\vert \\
&  -\left\vert \sin(\pi\rho\mathcal{L}_{j}\sin(\theta-\varphi_{j})-\pi
\rho\mathcal{L}_{j}|\sin\varphi_{j}|)\right\vert \\
&  >\frac{\gamma}{2}\left[  1-\frac{\gamma^{2}}{32}\right]  -\frac{\gamma}%
{4}\\
&  >\frac{\gamma}{5}\ .
\end{align*}

\noindent\emph{Case 2:} $|\sin(\pi\rho\mathcal{L}_{j}\sin\varphi_{j}%
)|\leq\gamma/2$.

\noindent Let $\rho$ be large so that $0\leq\theta\leq\frac{3}{2\pi
\rho\mathcal{L}_{j}}$ implies $\sin\theta\geq(1-\delta)\theta$, with
$\delta<1/20$. Then for
\begin{equation}
\frac{1}{\pi\rho\mathcal{L}_{j}}\leq\theta\leq\frac{3}{2\pi\rho\mathcal{L}%
_{j}} \label{newpo}%
\end{equation}
we have%
\[
\theta(1-\delta)\gamma+\left[  1-\frac{\theta^{2}}{2}\right]  |\sin\varphi
_{j}|\leq\sin\theta\cos\varphi_{j}-\cos\theta\sin\varphi_{j}\leq\theta
+|\sin\varphi_{j}|
\]
and
\begin{equation}
\theta\gamma(1-\delta)-\frac{\theta^{2}}{2}\leq\sin(\theta-\varphi_{j}%
)-|\sin\varphi_{j}|\leq\theta\ . \label{newsta}%
\end{equation}
For $\rho$ large enough we have $\frac{9}{8(\pi\rho\mathcal{L}_{j})^{2}}%
<\frac{\gamma\delta}{\pi\rho\mathcal{L}_{j}}$. Then (\ref{newpo}) and
(\ref{newsta}) yield%
\begin{align*}
\frac{\gamma(1-2\delta)}{\pi\rho\mathcal{L}_{j}}  &  <\frac{\gamma(1-\delta
)}{\pi\rho\mathcal{L}_{j}}-\frac{9}{8(\pi\rho\mathcal{L}_{j})^{2}}\leq
\theta\gamma\left(  1-\delta\right)  -\frac{\theta^{2}}{2}\\
&  <\sin(\theta-\varphi_{j})-|\sin\varphi_{j}|\leq\theta\leq\frac{3}{2\pi
\rho\mathcal{L}_{j}}%
\end{align*}
and
\begin{equation}
\gamma(1-2\delta)<\pi\rho\mathcal{L}_{j}\sin(\theta-\varphi_{j})-\pi
\rho\mathcal{L}_{j}|\sin\varphi_{j}|\leq\frac{3}{2}\ . \label{caso2tr}%
\end{equation}
We choose $\gamma$ small enough so that
\[
\sin(\gamma(1-2\delta))\geq(1-2\delta)^{2}\gamma\text{ \ \ \ and\ \ \ \ }%
\gamma^{2}/4<2\delta\ .
\]
Then (\ref{caso2tr}) and the assumption of \emph{Case 2} yield
\begin{align}
&  \left\vert \sin(\pi\rho\mathcal{L}_{j}\sin(\theta-\varphi_{j}))\right\vert
=\left\vert \sin(\pi\rho\mathcal{L}_{j}\left[  \sin(\theta-\varphi_{j}%
)-\sin\varphi_{j}\right]  +\pi\rho\mathcal{L}_{j}\sin\varphi_{j})\right\vert
\nonumber\\
&  \geq\left\vert \cos(\pi\rho\mathcal{L}_{j}\sin\varphi_{j})\right\vert
\left\vert \sin(\pi\rho\mathcal{L}_{j}\sin(\theta-\varphi_{j})-\pi
\rho\mathcal{L}_{j}|\sin\varphi_{j}|)\right\vert -\left\vert \sin(\pi
\rho\mathcal{L}_{j}\sin\varphi_{j})\right\vert \nonumber\\
&  \geq\gamma(1-2\delta)^{2}\sqrt{1-\frac{\gamma^{2}}{4}}-\frac{\gamma}%
{2}>\gamma\left[  (1-2\delta)^{5/2}-\frac{1}{2}\right]  >\frac{\gamma}%
{4}\ .\nonumber
\end{align}
\emph{Case 1} and \emph{Case 2} prove that for a suitable choice of
$0<\gamma<1$, such that $\cos\varphi_{j}>\gamma$, there exist $0<\alpha<\beta$
such that for $\frac{\alpha}{\pi\rho\mathcal{L}_{j}}\leq\theta\leq\frac{\beta
}{\pi\rho\mathcal{L}_{j}}$ and $\rho$ large enough we have
\begin{equation}
\left\vert \sin(\pi\rho\mathcal{L}_{j}\sin(\theta-\varphi_{j}))\right\vert
>\frac{\gamma}{5}\ . \label{stimaseno}%
\end{equation}
If $\varphi_{j}\in\lbrack-\pi,-\frac{\pi}{2})$ we have $\cos\varphi_{j}<0$ and
$\sin\varphi_{j}\leq0$. Then for $0\leq\theta<1$ we have
\[
-\theta+\left[  1-\frac{\theta^{2}}{2}\right]  |\sin\varphi_{j}|\leq\sin
\theta\cos\varphi_{j}-\cos\theta\sin\varphi_{j}\leq-\sin\theta|\cos\varphi
_{j}|+|\sin\varphi_{j}|\ .
\]
Hence, for a positive constant $K$,
\[
\sin\theta|\cos\varphi_{j}|\leq|\sin\varphi_{j}|-\sin(\theta-\varphi_{j})\leq
K\ \theta\ .
\]
If we choose a suitable constant $\gamma>0$ such that $|\cos\varphi
_{j}|>\gamma$, we can prove as for the case $\varphi_{j}\in(-\frac{\pi}{2},0]$
that (\ref{stimaseno}) still holds for $\frac{\alpha}{\pi\rho\mathcal{L}_{j}%
}\leq\theta\leq\frac{\beta}{\pi\rho\mathcal{L}_{j}}$, with $0<\alpha<\beta$
and $\rho$ large enough. Then (\ref{stimaseno}) yields%
\begin{align*}
&  \int_{0}^{2\pi}\left\vert F_{j}(\theta)\right\vert ^{2}d\theta\geq
\int_{\frac{\alpha}{\pi\rho\mathcal{L}_{j}}}^{\frac{\beta}{\pi\rho
\mathcal{L}_{j}}}\left\vert \frac{\sin(\pi\rho\ell_{j}\sin\theta)}{\sin\theta
}\cos\theta\sin(\pi\rho\mathcal{L}_{j}\sin(\theta-\varphi_{j}))\right\vert
^{2}d\theta\\
&  \geq c\ \gamma^{2}\int_{\frac{\alpha}{\pi\rho\mathcal{L}_{j}}}^{\frac
{\beta}{\pi\rho\mathcal{L}_{j}}}\left\vert \frac{\sin(\pi\rho\ell_{j}%
\sin\theta)}{\sin\theta}\right\vert ^{2}d\theta\geq c_{1}\ \rho\int_{c_{1}%
}^{c_{2}}\left\vert \frac{\sin(t)}{t}\right\vert ^{2}dt\geq c_{2}\ \rho.
\end{align*}
This ends the proof.
\end{proof}

The proof of Theorem \ref{maintheorem} will be complete after the proof of
Proposition \ref{propeps}. We need a simultaneous approximation lemma from
\cite{PS}.

\begin{lemma}
\label{diri} Let $r_{1},r_{2},...,r_{n}\in\mathbb{R}$. For every positive
integer $j$ there exists $j\leq q\leq j^{n+1}$ such that $\left\Vert
r_{s}q\right\Vert <j^{-1}$ for any $1\leq s\leq n$, where $\left\Vert
x\right\Vert $ denotes the distance of a real number $x$ from the integers.
\end{lemma}

\begin{proof}
[Proof of Proposition \ref{propeps}]Let $P\sim\{P_{j}\}_{j=1}^{2n}$ be a
polygon in $\mathfrak{P}$. For every positive integer $u$ let
\[
A_{u}^{j}:=\{k\in\mathbb{Z}^{2}:0<\mathcal{L}_{j}|k|\leq u^{2}%
\}\ \ \ \text{for }j=1,\ldots,n\ ,\ \ \ \ \ \ \ \ \text{\ }A_{u}%
:=\bigcup_{j=1}^{n}A_{u}^{j}\ .
\]
Observe that $\operatorname*{card}(A_{u}^{j})\leq4u^{4}$ and therefore
$\operatorname*{card}(A_{u})\leq4nu^{4}$. By Lemma \ref{diri} there exists a
sequence $\left\{  \rho_{u}\right\}  _{u=1}^{+\infty}$ of positive integers
such that, for$\ $every $k\in A_{u}$ and every $j=1,\ldots,n$,
\begin{equation}
u\leq\rho_{u}\leq u^{4nu^{4}+1}\ ,\ \ \ \ \ |\sin(\pi\rho_{u}|k|\mathcal{L}%
_{j})|<1/u\ . \label{i e ii}%
\end{equation}
Observe that (\ref{i e ii}) implies%
\begin{equation}
u\geq c_{\varepsilon}\log^{\frac{1}{4+\varepsilon}}(\rho_{u}) \label{iii}%
\end{equation}
for every $\varepsilon>0$. For any $1\leq j\leq n$ and $k\in A_{u}^{j}$ we
split the integral in (\ref{final}) into several parts.%
\[
E_{1,j,|k|}^{\rho}:=\int_{0}^{\left(  8\rho_{u}|k|\right)  ^{-1}}\left\vert
\frac{\sin(\pi\rho_{u}|k|\ell_{j}\sin\theta)}{\sin\theta}\sin(\pi\rho
_{u}|k|\mathcal{L}_{j}\cos\theta)\right\vert ^{2}d\theta\ .
\]
For $0\leq\theta\leq\left(  8\rho_{u}|k|\right)  ^{-1}$ we have $0\leq
1-\cos\theta\leq\left(  128\rho_{u}^{2}|k|^{2}\right)  ^{-1}$. Then
(\ref{i e ii}) yield%
\begin{align}
&  \left\vert \sin(\pi\rho_{u}|k|\mathcal{L}_{j}\cos\theta)\right\vert
\label{33}\\
&  =\left\vert \sin(\pi\rho_{u}|k|\mathcal{L}_{j}\left[  \cos\theta
-1+1\right]  )\right\vert \nonumber\\
&  \leq\left\vert \sin(\pi\rho_{u}|k|\mathcal{L}_{j}(\cos\theta-1))\cos
(\pi\rho_{u}|k|\mathcal{L}_{j})\right\vert \nonumber\\
&  +\left\vert \sin(\pi\rho_{u}|k|\mathcal{L}_{j})\cos(\pi\rho_{u}%
|k|\mathcal{L}_{j}(\cos\theta-1)\right\vert \nonumber\\
&  \leq\left\vert \sin(\pi\rho_{u}|k|\mathcal{L}_{j}(1-\cos\theta))\right\vert
+\left\vert \sin(\pi\rho_{u}|k|\mathcal{L}_{j})\right\vert \nonumber\\
&  \leq\frac{\pi\mathcal{L}_{j}}{128\rho_{u}|k|}+\left\vert \sin(\pi\rho
_{u}|k|\mathcal{L}_{j})\right\vert \nonumber\\
&  \leq c\ \frac{1}{u}\ .\nonumber
\end{align}
By (\ref{33}) we obtain
\begin{align*}
E_{1,j,|k|}^{\rho}  &  \leq c\ \frac{1}{u^{2}}\int_{0}^{\left(  8\rho
_{u}|k|\right)  ^{-1}}\left\vert \frac{\sin(\pi\rho_{u}|k|\ell_{j}\sin\theta
)}{\sin\theta}\right\vert ^{2}d\theta\\
&  \leq c_{1}\ \frac{|k|\rho_{u}}{u^{2}}\int_{0}^{1}\left\vert \frac{\sin
(t)}{t}\right\vert ^{2}dt\leq c_{2}\ \frac{|k|\rho_{u}}{u^{2}}\ .
\end{align*}
Let
\[
E_{2,j,|k|}^{\rho}:=\int_{\left(  8\rho_{u}|k|\right)  ^{-1}}^{\left(
8u^{1/4}\rho_{u}^{1/2}\left\vert k\right\vert ^{1/2}\right)  ^{-1}}\left\vert
\frac{\sin(\pi\rho_{u}|k|\ell_{j}\sin\theta)}{\sin\theta}\sin(\pi\rho
_{u}|k|\mathcal{L}_{j}\cos\theta)\right\vert ^{2}d\theta\ .
\]
For $\left(  8\rho_{u}|k|\right)  ^{-1}\leq\theta\leq\left(  8\ell^{1/4}%
\rho_{u}^{1/2}\left\vert k\right\vert ^{1/2}\right)  ^{-1}$ we have
\[
\frac{1}{2000\rho_{u}^{2}|k|^{2}}\leq2\sin^{2}\left(  \theta/2\right)
=1-\cos\theta\leq\frac{1}{128u^{1/2}\rho_{u}|k|}\ .
\]
As in (\ref{33}) we obtain
\begin{align*}
\left\vert \sin(\pi\rho_{u}|k|\mathcal{L}_{j}\cos\theta)\right\vert  &
\leq\left\vert \sin(\pi\rho_{u}|k|\mathcal{L}_{j}(1-\cos\theta))\right\vert
+\left\vert \sin(\pi\rho_{u}|k|\mathcal{L}_{j})\right\vert \\
&  \leq\frac{\pi\mathcal{L}_{j}}{128u^{1/2}}+\frac{1}{u}\leq c\ u^{-1/2}%
\end{align*}
and then
\begin{align*}
E_{2,j,|k|}^{\rho}  &  \leq c\ \frac{1}{u}\int_{\left(  8\rho_{u}|k|\right)
^{-1}}^{\left(  8u^{1/4}\rho_{u}^{1/2}\left\vert k\right\vert ^{1/2}\right)
^{-1}}\left\vert \frac{\sin(\pi\rho_{u}|k|\ell_{j}\sin\theta)}{\sin\theta
}\right\vert ^{2}d\theta\\
&  \leq c_{1}\ \frac{1}{u}\int_{\left(  8\rho_{u}|k|\right)  ^{-1}}^{\left(
8u^{1/4}\rho_{u}^{1/2}\left\vert k\right\vert ^{1/2}\right)  ^{-1}}%
\frac{d\theta}{\theta^{2}}\leq c_{2}\ \frac{\rho_{u}|k|}{u}\ .
\end{align*}
Let $1/4<\lambda<1/2$ and let
\[
E_{3,j,|k|}^{\rho}:=\int_{\left(  8u^{1/4}\rho_{u}^{1/2}\left\vert
k\right\vert ^{1/2}\right)  ^{-1}}^{\lambda}\left\vert \frac{\sin(\pi\rho
_{u}|k|\ell_{j}\sin\theta)}{\sin\theta}\sin(\pi\rho_{u}|k|\mathcal{L}_{j}%
\cos\theta)\right\vert ^{2}d\theta\ .
\]
We have%
\[
E_{3,j,|k|}^{\rho}\leq\int_{\left(  8u^{1/4}\rho_{u}^{1/2}\left\vert
k\right\vert ^{1/2}\right)  ^{-1}}^{\lambda}\frac{d\theta}{\theta^{2}}%
\leq8u^{1/4}\rho_{u}^{1/2}\left\vert k\right\vert ^{1/2}\ .
\]
Finally we have
\[
E_{4,j,|k|}^{\rho}:=\int_{\lambda}^{\frac{\pi}{2}}\left\vert \frac{\sin
(\pi\rho_{u}|k|\ell_{j}\sin\theta)}{\sin\theta}\sin(\pi\rho_{u}|k|\mathcal{L}%
_{j}\cos\theta)\right\vert ^{2}d\theta\leq c\ .
\]
By the above estimates, (\ref{final}), (\ref{i e ii})\ and (\ref{iii}) we have%
\begin{align*}
&  \Vert D_{P}^{\rho_{u}}\Vert_{L^{2}\left(  SO(2)\times\mathbb{T}^{2}\right)
}^{2}\\
&  \leq c\ \rho_{u}\sum_{k\in A_{u}}\frac{1}{|k|^{3}}\left(  \frac{1}{u^{2}%
}+\frac{1}{u}+u^{1/4}\rho_{u}^{-1/2}\left\vert k\right\vert ^{-1/2}+\rho
_{u}^{-1}\left\vert k\right\vert ^{-1}\right) \\
&  +c_{1}\ \sum_{k\notin A_{u}}\frac{1}{|k|^{4}}\int_{0}^{\pi/2}\left\vert
\frac{\sin(\pi\rho_{u}|k|\ell_{j}\sin\theta)}{\sin\theta}\right\vert
^{2}\ d\theta\\
&  \leq c\ \rho_{u}\sum_{0\neq k\in A_{u}}\frac{1}{|k|^{3}}u^{-1/4}\\
&  +c_{1}\ \sum_{\left\vert k\right\vert >c_{1}u^{2}}\frac{1}{|k|^{4}}\left(
\int_{0}^{\left(  \rho_{u}\left\vert k\right\vert \right)  ^{-1/2}}\left(
\rho_{u}\left\vert k\right\vert \right)  \ d\theta+\int_{\left(  \rho
_{u}\left\vert k\right\vert \right)  ^{-1/2}}^{\pi/2}\frac{1}{\theta^{2}%
}\ d\theta\right) \\
&  \leq c_{\varepsilon}\ \rho_{u}\sum_{0\neq k\in\mathbb{Z}^{2}}\frac
{1}{|k|^{3}}\log^{-\frac{1}{16+\varepsilon}}\left(  \rho_{u}\right)
+c\ \sum_{\left\vert k\right\vert >c_{1}u^{2}}\frac{1}{|k|^{4}}\left(
\rho_{u}\left\vert k\right\vert \right)  ^{1/2}\\
&  \leq c_{\varepsilon}\ \rho_{u}\log^{-\frac{1}{16+\varepsilon}}\left(
\rho_{u}\right)  +c\ \rho_{u}^{1/2}\int_{\{t\in\mathbb{R}^{2}:|t|>c_{1}%
u^{2}\}}\frac{1}{\left\vert t\right\vert ^{7/2}}\ dt\\
&  \leq c_{\varepsilon}\ \rho_{u}\log^{-\frac{1}{16+\varepsilon}}\left(
\rho_{u}\right)  \ .
\end{align*}

\end{proof}

We now turn to the proof of Proposition \ref{stimaepsilon}, which depends on
the following lemma proved by L. Parnovski and A. Sobolev \cite{PS}.

\begin{lemma}
\label{PSlemma} For any $\varepsilon>0$ there exist $\rho_{0}\geq1$ and
$0<\alpha<1/2$ such that for every $\rho\geq\rho_{0}$ there exists
$k\in\mathbb{Z}^{d}$ such that $|k|\leq\rho^{\varepsilon}$ and $\Vert
\rho|k|\Vert\geq\alpha$, where $\Vert x\Vert$ is the distance of a real number
$x$ from the integers.
\end{lemma}

\begin{proof}
[Proof of Proposition \ref{stimaepsilon}]Let $P\sim\{P_{j}\}_{j=1}^{2n}$ be a
polygon in $\mathfrak{P}$. Let $\varepsilon>0$ and let $j\in\left\{
1,2,\ldots,n\right\}  $. By Lemma \ref{PSlemma} there exist $\rho_{0}\geq1$
and $0<a<1/2$ such that for any $\rho\geq\rho_{0}$ there is $\widetilde{k}%
\in\mathbb{Z}^{2}$ such that $|\widetilde{k}|\leq\rho^{\varepsilon/3}$ and
$|\sin(\pi\rho|\widetilde{k}|\mathcal{L}_{j})|>a$. Then we consider the
interval%
\begin{equation}
\theta_{j}\leq\theta\leq\theta_{j}+\frac{1}{\pi\rho|\widetilde{k}|}\ .
\label{interval j}%
\end{equation}
We have
\[
0\leq1-\cos(\theta-\theta_{j})\leq\frac{1}{2(\pi\rho|\widetilde{k}|)^{2}}\ .
\]
Then for large $\rho$ we have
\begin{align}
&  |\sin(\pi\rho|\widetilde{k}|\mathcal{L}_{j}\cos(\theta-\theta
_{j}))|\label{epsi1}\\
&  \geq|\sin(\pi\rho|\widetilde{k}|\mathcal{L}_{j})||\cos(\pi\rho
|\widetilde{k}|\mathcal{L}_{j}(1-\cos(\theta-\theta_{j})))|\nonumber\\
&  -|\sin(\pi\rho|\widetilde{k}|\mathcal{L}_{j}(1-\cos(\theta-\theta
_{j})))|\nonumber\\
&  \geq c\ \left[  1-\frac{\mathcal{L}_{j}^{2}}{8(\pi\rho|\widetilde{k}|)^{2}%
}\right]  -\frac{\mathcal{L}_{j}}{2\pi\rho|\widetilde{k}|}>c_{1}\ .\nonumber
\end{align}
As before the sides non parallel to $P_{j}P_{j+1}$ give a bounded contribution
to the integration of $\left\vert D_{P_{n}}^{\rho}\right\vert ^{2}$ over the
interval in (\ref{interval j}). Finally (\ref{epsi1}) yields%
\begin{align*}
&  \Vert D_{P_{n}}^{\rho}\Vert_{L^{2}\left(  SO(2)\times\mathbb{T}^{2}\right)
}^{2}\\
&  \geq c+c_{1}\ \frac{1}{|\widetilde{k}|^{4}}\int_{\theta_{j}}^{\theta
_{j}+1/(\pi\rho|\widetilde{k}|)}\left\vert \frac{\sin(\pi\rho|\widetilde
{k}|\ell_{j}\sin(\theta-\theta_{j}))}{\sin(\theta-\theta_{j})}\right\vert
^{2}\\
&  \times|\sin(\pi\rho|\widetilde{k}|\mathcal{L}_{j}\cos(\theta-\theta
_{j}))|^{2}|\cos(\theta-\theta_{j})|^{2}\ d\theta\\
&  \geq c+c_{2}\ \frac{1}{|\widetilde{k}|^{4}}\int_{0}^{1/(\pi\rho
|\widetilde{k}|)}\left\vert \frac{\sin(\pi\rho|\widetilde{k}|\ell_{j}%
\sin\theta)}{\sin\theta}\right\vert ^{2}\ d\theta\\
&  \geq c+c_{3}\frac{\rho}{|\widetilde{k}|^{3}}\\
&  \geq c_{4}\ \rho^{1-\varepsilon}\ .
\end{align*}

\end{proof}

The proofs of Lemmas \ref{lemma1}, \ref{lemma2}, and \ref{lemma3} actually
show that $\left\Vert \widehat{\chi_{P}}\left(  \rho\cdot\right)  \right\Vert
_{L^{2}\left(  \Sigma_{1}\right)  }\geq c\ \rho^{-3/2}$ whenever
$P\notin\mathfrak{P}$. Hence Theorem \ref{maintheorem} and Proposition
\ref{prop} readily yield the following result.

\begin{corollary}
Let $P$ be a polygon in $\mathbb{R}^{2}$. Then $P$ satisfies%
\[
\left\Vert \widehat{\chi_{P}}\left(  \rho\cdot\right)  \right\Vert
_{L^{2}\left(  \Sigma_{1}\right)  }\geq c\ \rho^{-3/2}%
\]
if and only if $P\notin\mathfrak{P}$.
\end{corollary}

\bigskip

The results in this paper (apart from Lemma \ref{lemma1}) seem to be tailored
for the planar case. A different (perhaps simpler) approach might be necessary
in order to deal with the multi-dimensional cases.

\bigskip

Giancarlo Travaglini

Dipartimento di Statistica e Metodi Quantitativi

Edificio U7, Universit\`{a} di Milano-Bicocca

Via Bicocca degli Arcimboldi 8, 20126, Milano, Italia

\textrm{giancarlo.travaglini@unimib.it}

\medskip

Maria Rosaria Tupputi,

Dipartimento di Statistica e Metodi Quantitativi

Edificio U7, Universit\`{a} di Milano-Bicocca,

Via Bicocca degli Arcimboldi 8, 20126 Milano, Italia

\textrm{maria.tupputi@unimib.it}
\end{document}